\pdfoutput=1 
\RequirePackage[l2tabu, orthodox]{nag} 

\documentclass[12pt, reqno]{amsart}

\usepackage{amssymb, amsmath, amsthm}

\usepackage[all,warning]{onlyamsmath} 

\usepackage[usenames,dvipsnames]{xcolor}

\usepackage[backref, colorlinks=true, citecolor=Green]{hyperref}
\usepackage[alphabetic,backrefs,lite]{amsrefs}
\usepackage{amscd}   
\usepackage{fullpage}
\usepackage{colonequals}
\usepackage{enumerate}
\usepackage{tikz-cd} 
\usepackage[all,cmtip]{xy}
\usepackage{enumitem}

\usepackage{bm}
\usepackage{comment}

\DeclareFontEncoding{OT2}{}{} 

\newtheorem{lemma}{Lemma}[section]
\newtheorem{theorem}[lemma]{Theorem}
\newtheorem{proposition}[lemma]{Proposition}
\newtheorem{prop}[lemma]{Proposition}
\newtheorem{cor}[lemma]{Corollary}

\newtheorem{claim*}{Claim}
\newtheorem{thm}[lemma]{Theorem}

\theoremstyle{remark}
\newtheorem{remark}[lemma]{Remark}
\newtheorem{remarks}[lemma]{Remarks}

\newcommand{\A}{{\mathbb A}}
\newcommand{\Aff}{{\mathbb A}}
\newcommand{\G}{{\mathbb G}}

\newcommand{\PP}{{\mathbb P}}
\newcommand{\pp}{{\mathbb P}}

\newcommand{\F}{{\mathbb F}}

\newcommand{\Q}{{\mathbb Q}}

\newcommand{\Z}{{\mathbb Z}}

\newcommand{\kbar}{{\overline{k}}}

\newcommand{\Fbar}{{\overline{\F}}}

\newcommand{\kk}{{\mathbf k}}

\newcommand{\calE}{{\mathcal E}}

\newcommand{\calG}{{\mathcal G}}

\newcommand{\calO}{{\mathcal O}}

\newcommand{\calU}{{\mathcal U}}

\newcommand{\OO}{{\mathcal O}}

\newcommand{\frakp}{{\mathfrak p}}

\usepackage[mathscr]{euscript}

\DeclareMathOperator{\HH}{H}

\DeclareMathOperator{\Char}{char}
\DeclareMathOperator{\inv}{inv}

\DeclareMathOperator{\im}{im}

\DeclareMathOperator{\Hom}{Hom}

\DeclareMathOperator{\Gal}{Gal}

\DeclareMathOperator{\Cor}{Cor}

\DeclareMathOperator{\Br}{Br}

\DeclareMathOperator{\divv}{div}

\DeclareMathOperator{\Pic}{Pic}

\DeclareMathOperator{\Spec}{Spec}

\DeclareMathOperator{\Proj}{Proj}

\DeclareMathOperator{\sep}{sep}

\DeclareMathOperator{\et}{et}

\DeclareMathOperator{\Span}{Span}

\DeclareMathOperator{\spl}{split}

\newcommand{\del}{\partial}
\newcommand{\eps}{\varepsilon}
\newcommand{\isom}{\simeq}

\newcommand{\injects}{\hookrightarrow}
\newcommand{\pairing}{\langle {\cdot}, {\cdot} \rangle}
\newcommand{\surjects}{\twoheadrightarrow}
\newcommand{\To}{\longrightarrow}

\newcommand{\intersect}{\cap} 
\newcommand{\Intersection}{\bigcap} 
\newcommand{\Union}{\bigcup} 

\renewcommand{\setminus}{\smallsetminus} 

\numberwithin{equation}{section}
\numberwithin{table}{section}

\newcommand{\defi}[1]{\textsf{#1}} 

\title{Brauer--Manin obstructions requiring arbitrarily many Brauer classes}

\author{Jennifer Berg}
\address{Bucknell University, Department of Mathematics, Lewisburg, PA 17837, USA}
\email{jsb047@bucknell.edu}
\urladdr{\url{https://sites.google.com/view/jenberg}}

\author{Carlo Pagano}
\address{Concordia University, Department of Mathematics and Statistics, Montreal, Quebec H3G 1M8, Canada}
\email{carlein90@gmail.com}
\urladdr{\url{https://sites.google.com/view/carlopagano}}

\author{Bjorn Poonen}
\address{Department of Mathematics, Massachusetts Institute of Technology, Cambridge, MA 02139-4307, USA}
\email{poonen@math.mit.edu}
\urladdr{\url{https://math.mit.edu/~poonen}}

\author{Michael Stoll}
\address{Mathematisches Institut, Universit\"at Bayreuth,
         95440 Bayreuth, Germany}
\email{Michael.Stoll@uni-bayreuth.de}
\urladdr{\url{https://www.mathe2.uni-bayreuth.de/stoll/}}
         
\author{Nicholas Triantafillou}
\address{Center for Communications Research, 805 Bunn Drive, Princeton, NJ 08540, USA}
\email{n.triantafillou@idaccr.org}
\urladdr{\url{https://ngtriant.github.io/}}

\author{Bianca Viray}
\address{University of Washington, Department of Mathematics, Box 354350, Seattle, WA 98195, USA}
\email{bviray@uw.edu}
\urladdr{\url{https://math.washington.edu/~bviray}}

\author{Isabel Vogt}
\address{Brown University, Department of Mathematics, Box 1917, 151 Thayer Street, Providence, RI 02912, USA}
\email{ivogt.math@gmail.com}
\urladdr{\url{https://www.math.brown.edu/ivogt/}}

\date{August 15, 2023}

\begin{document}

	 \begin{abstract}
  On a projective variety defined over a global field, any Brauer--Manin obstruction to the existence of rational points is captured by a finite subgroup of the Brauer group.  
  We show that this subgroup can require arbitrarily many generators.
	 \end{abstract}

	\maketitle

\section{Introduction}
    A fundamental problem in arithmetic geometry is to determine, given a regular projective geometrically integral variety \(X\) over a global field \(k\), whether \(X(k)\neq\emptyset\).
    A first approach is to check local solvability, i.e., whether \(X(k_v)\neq\emptyset\) for all places \(v\); this can be done effectively.  However, there are many varieties that fail the so-called local-to-global principle; that is, \(X(k) = \emptyset\) even though \(X(k_v)\neq\emptyset\) for all places \(v\).  All but a handful of such examples in the literature are explained by the \defi{Brauer--Manin obstruction}, first introduced by Manin in 1970~\cite{Manin-ICM}.  

    Manin observed that the functoriality of the Brauer group gives a pairing 
    \[
        \pairing \colon X(\A_k)\times \Br X\to \Q/\Z  
    \]
    and an exact sequence from global class field theory \eqref{equation:global class field theory} implies that \(X(k)\) is contained in the left kernel of $\pairing$.  More precisely, each element \(\alpha\in \Br X\) defines a subset \(X(\A_k)^\alpha\) consisting of all adelic points orthogonal to \(\alpha\) under $\pairing$.
    For $B \subset \Br X$, define \(X(\A_k)^{B} \colonequals \Intersection_{\alpha \in B} X(\A_k)^\alpha\).
    The \defi{Brauer--Manin set} is \(X(\A_k)^{\Br} \colonequals X(\A_k)^{\Br X}\).

    For a regular projective variety \(X\), the continuity of $\pairing$ implies that if \(X(\A_k)^{\Br} = \emptyset\) then there exists a \emph{finite} subgroup \(B\subset \Br X\) that captures the obstruction, in the sense that \(X(\A_k)^{B}=\emptyset\).  In fact, in the literature, the vast majority of examples with \(X(\A_k)^{\Br} = \emptyset\) have \(X(\A_k)^{\langle \alpha\rangle}\) for one \(\alpha \in \Br X\), but this is likely because it is simpler to compute the obstruction from one element. To the best of our knowledge, the only examples requiring more than one element are~\cite{KT-dp2}*{Example 7} and~\cite{Corn}*{Example 9.4}.
    
    Our main result shows that the examples in the literature are not representative of what is theoretically possible, that the finite subgroup \(B\subset \Br X\) that captures a Brauer--Manin obstruction can require arbitrarily many generators.  
    
    \begin{theorem}\label{thm:Main}
        Let \(N\) be a positive integer and let \(k\) be a global field of characteristic not \(2\).  There exists a smooth projective geometrically integral variety $X$ over $k$ such that \hbox{\(X(\A_k)^{\Br} = \emptyset\)}, while for all subgroups \(B\subset \Br X\) generated by at most $N$ elements, \(X(\A_{k})^{B}\) is nonempty.
    \end{theorem}

    \begin{remarks}\hfill
        \begin{enumerate}[label*=(\alph*), font=\normalfont]
        \item 
        If $S$ is a subset of $\Br X$ and $\langle S \rangle$ is the group generated by $S$, then \(X(\A_k)^{S} = X(\A_k)^{\langle S \rangle}\).  This is why we use only \emph{subgroups} of the Brauer group in our statements.
        \item The variety \(X\) we construct is a conic bundle over \(\mathbb{P}^1\). In particular, \(X\) is a geometrically rational surface.  However, for some other types of geometrically rational surfaces, a Brauer--Manin obstruction, when present, is always captured by a single element of the Brauer group.  For example, this holds for quartic and cubic del Pezzo surfaces~\cite{CTP}*{Lemma 3.4 and the following Remark (ii)}.
        \item Theorem~\ref{thm:Main} considers how many generators are required for the subgroup \(B\).  One can also ask about constraints on the orders of the elements in \(B\).  This and related questions were considered in work of Skorobogatov and Zarhin~\cite{SZ-Kummer2torsion}, Creutz and Viray~\cite{CV-Capturing}, Creutz, Viray, and Voloch~\cite{CVV-Capturing}, and Nakahara~\cite{Nakahara-Capturing}. See also~\cite{Viray-PCMI} for a survey of these questions. 
        \item In this paper, we are interested in varieties with empty Brauer--Manin set.  It is worth noting that there are varieties with nonempty Brauer--Manin set that still fail to have rational points.  The first example of this was constructed by Skorobogatov in 1999~\cite{Skorobogatov-etBr}; see~\cites{Poonen-Insufficiency, HS, BBMPV, CTPS, KPS} for further examples.
        \end{enumerate}
    \end{remarks}

\subsection{Outline}
In~\cite{CTP}*{Lemma 3.4}, Colliot-Th\'el\`ene and Poonen reformulate the Brauer--Manin obstruction in terms of dual groups.  We use this perspective in Section~\ref{sec:group} to provide a combinatorial local criterion for the behavior in Theorem~\ref{thm:Main}.  The remainder of the paper focuses on constructing a variety \(X\) which exhibits this local behavior.  In Section~\ref{sec:poly}, we collect various lemmas about polynomials taking prescribed values which will be used in the proof.  In Section~\ref{S:Faddeev}, we present a version of Faddeev's theorem on the Brauer group of a rational function field that is valid over any constant field.  In Section~\ref{S:generalities}, we review and generalize some well-known facts on conics and conic bundles; in particular, our Theorem~\ref{theorem:BrConicBundle} generalizes a result on Brauer groups of conic bundles to a setting more general than we need, for the sake of potential future applications. The heart of the paper is Section~\ref{S:base change}, which constructs suitable conic bundles split by a constant extension by using a pullback construction generalizing~\cite{CTP}*{Lemma 3.3}. In Section~\ref{sec:explicit}, we give for each $N \ge 1$, an explicit construction of a conic bundle over \(\Q\) as in Theorem~\ref{thm:Main} with the additional property that the local evaluation maps \(X(\Q_v)\times \Br X\to \Q/\Z\) are constant for all but one place \(v\).  Finally, we include an appendix which explores the possible ways that the combinatorial criterion from Section~\ref{sec:group} can be achieved.

\section*{Acknowledgements}
    This project started at the Park City Mathematics Institute (PCMI) 2022 program ``Number theory informed by computation''.  We thank the PCMI director, Rafe Mazzeo, the PCMI staff, particularly Dena Vigil, and the PCMI scientific committee for their work that made such a program possible.  We also thank the program organizers, Jennifer Balakrishnan, Bjorn Poonen, and Akshay Venkatesh, for giving us the opportunity to participate.

    We thank Eric Larson for helpful conversations around the topic of Proposition~\ref{prop:SumsetContainsSubspacenew}
    and Fedor Petrov for pointing us to Rado's theorem when we asked on \href{https://mathoverflow.net/questions/443133/linearly-independent-vectors-from-a-family-of-subspaces}{MathOverflow} for a reference.

    This material is based partially upon work supported by National Science Foundation grant DMS-1928930 while several of the authors (J.B., M.S., B.V., and I.V.) were in residence at the Simons Laufer Mathematical Sciences Institute in Berkeley, California, during the Spring 2023 semester. 
C.P. thanks the Max Planck Institute in Bonn for hospitality and excellent work conditions. 
B.P. was supported in part by NSF grant DMS-2101040 and Simons Foundation grants \#402472 and \#550033.
M.S. was supported in part by DFG grants STO 299/13-1 (AOBJ: 635353) and STO 299/17-1 (AOBJ: 662415).
    B.V. was supported in part by an AMS Birman Fellowship and NSF grant DMS-2101434. I.V. was supported in part by NSF grant DMS-2200655. 

\section{Notation} 

For any field $k$, let $\Br k$ be its Brauer group.
If $\Char k \ne 2$ and $a,b \in k^\times$, then let $(a,b)$ be the class of the quaternion algebra $k\langle i,j \rangle/(i^2=a,j^2=b,ij=-ji)$ in $\Br k$.

Suppose that \(k\) is a global field.
For each place \(v\) of \(k\), let $k_v$ be the completion, 
and let $\inv_v \colon \Br k_v \to \Q/\Z$ be the injection defined using the sign convention of \cite{CTS-BrauerBook}*{Def.~13.1.7}; see also \cite{CTS-BrauerBook}*{Rmk.~13.1.12}.
There is an exact sequence
\begin{equation}
\label{equation:global class field theory}
0 \To \Br k \To \bigoplus_v \Br k_v \xrightarrow{\sum_v\inv_v} \Q/\Z \To 0.
\end{equation}

For any scheme $X$, let $X^{(1)}$ be the set of points of codimension~$1$.  If $t \in X$, let $\kk(t)$ be the residue field.  If $X$ is integral, let $\kk(X)$ be the function field of $X$.

Let \(X\) be a regular projective geometrically integral variety over $k$.  Its \defi{Brauer group} is \(\Br X \colonequals \HH^2_{\et}(X, \G_m)\), and its subgroup of \defi{constant classes} is \(\Br_0 X \colonequals \im\left(\Br k \to \Br X\right)\). We write \(\overline{\Br}\,X\) for the quotient $\Br X/\Br_0 X$.
Summing the local pairings
\begin{align*}
    \pairing_v \colon X(k_v) \times \Br X &\To \Q/\Z \\
    ( P_v , \alpha) &\longmapsto \inv_v P_v^* \alpha
\end{align*}
defines the \defi{Brauer--Manin pairing} 
\( \pairing \colon X(\A_k) \times \Br X \to \Q/\Z\).
Fix a finite subgroup $B \subset \Br X$ and let $\hat{B} = \Hom(B, \Q/\Z)$; 
we obtain set maps
\begin{align*}
    \phi_v \colon X(k_v) &\To \hat{B} \\
    P_v &\longmapsto (\alpha \mapsto \inv_v P_v^* \alpha )
\end{align*}
summing to \(\phi \colon X(\A_k)\to \hat{B}\),
and their images $S_v \colonequals \phi_v(X(k_v))$ sum in $\hat{B}$ to the image $S \colonequals \phi(X(\A_k))$ (we have \(S_v = \{0\}\) for all but finitely many places \(v\), so the sum is well-defined).
These definitions extend naturally to an injective homomorphism $B \hookrightarrow \Br X$ (rather than a subgroup \(B\subset \Br X\)).

\section{A group-theoretic interpretation of the Brauer--Manin obstruction} \label{sec:group}
    
    \begin{lemma} Let \(X/k\) be a regular projective geometrically integral variety and let \(B \subset \Br X\) be a finite subgroup.
        \begin{enumerate}[label*=\normalfont(\alph*)]
            \item \label{I:0 not in S}
            We have \(X(\A_k)^B = \emptyset\) if and only if \(0\notin S\).
            \item There is a proper subgroup \(B'\subsetneq B\) with \(X(\A_k)^{B'} = \emptyset\) if and only if there exists a nontrivial subgroup \({H}\subseteq \hat{B}\) that is disjoint from \(S\).
        \end{enumerate}
    \end{lemma}
     
    \begin{proof}
    \hfill
    \begin{enumerate}[label*=\normalfont(\alph*)]
    \item
    For $y \in X(\A_k)$, the following three conditions are equivalent:
    $y \in X(\A_k)^{B}$; \hbox{$\langle y,B \rangle = 0$}; $\phi(y)=0$.
    Thus $X(\A_k)^{B} \ne \emptyset$ if and only if $0 \in \phi(X(\A_k)) \equalscolon S$.
        \item
        Suppose that $B'$ and $H$ are groups corresponding under the bijection
       \begin{align*}
        \{\textup{subgroups of $B$}\} &\longleftrightarrow \{\textup{subgroups of $\hat{B}$}\}\\
        B' &\longmapsto H \colonequals \ker(\hat{B}\to \hat{B'}).
        \end{align*}
        Then $B' \ne B$ if and only if $H \ne \{0\}$.
        Let $S'$ be the image of the composition \(X(\A_k) \to \hat{B} \to \hat{B'}\).
        By~\ref{I:0 not in S},  \(X(\A_k)^{B'} = \emptyset\) if and only if $0 \notin S'$,
        which holds if and only if $H \cap S = \emptyset$, since $H = \ker(\hat{B}\to\hat{B'})$ and $S'$ is the image of $S$ under $\hat{B}\to\hat{B'}$.\qedhere
        \end{enumerate}
    \end{proof}

    \begin{cor}\label{cor:2torsion}
    Let \(B\subset \Br X\) be a finite elementary abelian \(2\)-group.
    Then the two conditions \(X(\A_k)^B = \emptyset\) and \(X(\A_k)^{B'} \neq \emptyset\) for all proper subgroups \(B'\subsetneq B\) hold if and only if \(S = \hat{B}\setminus\{0\}\).
    \end{cor}

To prove Theorem~\ref{thm:Main}, we construct a variety~$X$ whose Brauer group contains a subgroup \(B\) isomorphic to \((\Z/2\Z)^n\) such that the induced map \(B \to \overline{\Br}\, X\) is an isomorphism and the image $S$ of \(X(\A_k)\) in \(\hat B\) is $\hat{B}\setminus\{0\}$.
One potential way to obtain \(S = \hat{B} \setminus \{0\}\) is to arrange that \(S_{v_0} = \hat{B} \setminus \{0\}\) for a single place \(v_0\) and \(S_v = \{0\}\) for all \(v\neq v_0\); this is what we do in Section~\ref
{S:base change} after some preliminaries in Sections \ref{sec:poly}, \ref{S:Faddeev}, and~\ref{S:generalities}.
There are other ways to arrange $S = \hat{B}\setminus\{0\}$, but the number of places \(v\) where \(\phi_v\) is nonconstant can be bounded in terms of \(\dim_{\F_2} B\):
\begin{theorem}[Specialization of Theorem~\ref{thm:sharpboundnew}]\label{thm:sharpboundBr}
    Let \(B \subset\Br X\) be a finite elementary abelian \(2\)-group.  If \(X(\A_k)^B = \emptyset\) and \(X(\A_k)^{B'}\neq\emptyset\) for all proper subgroups \(B'\subsetneq B\), then
    \[
        \#\{\textup{places $v$ of $k$} : \#S_v \ge 2\} \leq \dim_{\F_2} B - 1.
    \]
\end{theorem}
\noindent In the appendix, we also show that, set-theoretically, this upper bound is sharp.

\section{Polynomials with prescribed values}
\label{sec:poly}
\subsection{Square and nonsquare values of polynomials}

In this section we collect various lemmas about polynomials taking prescribed values which will be used
in the proof of Theorem~\ref{thm:Main}.
\begin{proposition}
\label{P:Weil bound}
Let \(\F\) be a finite field of odd characteristic,
\(n\) be a nonnegative integer, 
$f_1,\ldots,f_n \in \F[x]$, and set $d = \sum \deg f_i$.
Assume that $f_1,\ldots,f_n$ are independent in $\Fbar(x)^\times/\Fbar(x)^{\times 2}$.
Then
\[
    \#\{c\in \F : f_i(c)\varepsilon_i\in \F^{\times2} \textup{ for all }i\} = (\#\F)/2^n + O(d \sqrt{\#\F}),
\]
for any choice of \(\varepsilon_1,\ldots, \varepsilon_n\in \F^{\times}\), where the implicit constant in~$O(d\sqrt{\#\F})$ is absolute.
In particular, if $\#\F$ is sufficiently large relative to $n$ and $d$, then the polynomials \(f_i\) simultaneously realize all collections of square classes \((\F^{\times}/\F^{\times2})^n\).
\end{proposition}

\begin{proof}
For each $i$, let $Z_i \subset \Aff^1$ be the set of zeros of $f_i$.
Let $Z=\Union Z_i$, so $\#Z \le d$.
Let $U=\Aff^1\setminus Z$.
Let $X$ be the $\{\pm 1\}^n$-cover of $U$ defined by the equations $\varepsilon_i y_i^2 = f_i(x)$ for $i=1,\ldots,n$.
The independence hypothesis implies that $X$ is geometrically irreducible.
By the Riemann--Hurwitz formula, $X$ is a curve of genus $O(2^n d)$ minus at most $2^n (\#Z+1) = O(2^n d)$ punctures (the $+1$ is for $\infty$).
Now
\[
   2^n \#\{c\in \F : f_i(c)\varepsilon_i\in \F^{\times 2} \textup{ for all }i\} = \#X(\F) = \#\F + O(2^n d \sqrt{\#\F}),
\]
by the Weil bound.
The two claims follow.
\end{proof}

\subsection{Polynomial interpolation}

\begin{lemma} \label{basic}\label{lem:PolynomialInterpolation}
Let \(\F\) be a finite field, let \(n \geq \#\F\) be an integer,
and let $g\colon \F \to \F$ be a function. Then there exists a monic degree $n$ polynomial $f\in\F[x]$ such that $f(c)=g(c)$ for each $c\in\F$.
\end{lemma}
\begin{proof}
Lagrange interpolation constructs $h(x)$ of degree at most $\#\F-1$ with the same values as $g(x) - x^n$.
Let $f(x) = x^n + h(x)$.
\end{proof}

\begin{lemma} 
\label{L:special polynomial}
Let \(v\) be a nonarchimedean place of \(k\).
Let $\calU_1,\ldots,\calU_m$ be nonempty open subsets of $\PP^1(k_v)$ whose union $\calU$ contains $\infty$.
For any sufficiently divisible positive integer $d$, 
there exists a degree~$d$ polynomial $g \in k_v[x]$
such that $g(\PP^1(k_v))$ is contained in $\calU$ and meets every~\(\calU_j\).
\end{lemma}

\begin{proof} 
Write $\calO$ for the valuation ring of \(k_v\), write $\frakp$ for its maximal ideal, and \(q \colonequals \#\calO/\frakp\).
By repeating the \(\calU_j\) if necessary, we may enlarge $m$ to assume that \(m=q^{e}\) for some positive integer \(e\).
Let $a_1,\ldots,a_m \in \calO$ be a complete set of representatives for $\calO/\frakp^{e}$.  
For $i=1,\ldots,m$, define 
\[
    h_i(x) \colonequals \prod_{j \ne i} \frac{(x-a_j)}{(a_i - a_j)}.
\]

\emph{Claim~1: If \(b \in a_i + \frakp^{e}\), then $h_i(b) \in 1 + \frakp$.}

\emph{Proof}: For each \(j\neq i\), the factor $\frac{b - a_j}{a_i - a_j} = 
1 + \frac{b - a_i}{a_i - a_j}$ belongs to $1 + \frakp$.

\medskip

\emph{Claim~2: If $b \in a_\ell + \frakp^e$ for some $\ell \ne i$, then $h_i(b) \in \frakp$.}

\emph{Proof}: We have $v(\prod_{j \ne i,\ell} (b-a_j)) = v(\prod_{j \ne i,\ell} (a_j-a_i))$ because the factors in each product are representatives for all cosets except $\frakp^e$ and $a_\ell-a_i + \frakp^e$,
and representatives for the same nonzero coset have the same valuation.
Thus $v(h_i(b)) = v( \frac{b-a_\ell}{a_i-a_\ell} ) \ge e - v(a_i - a_{\ell}) > 0$.

\medskip

For each $i$, choose $c_i \in \calU_i - \{1,\infty\}$. 
Let $M$ be a positive integer.
Now set 
\[
    g(x) = \prod_{i=1}^m (1+(c_i-1)h_i(x)^M).
\]
If $M$ is divisible by a large enough power of $q$, 
then $g$ maps $a_i + \frakp^e$ into $\calU_i$ 
and maps $\PP^1(k_v)-\calO$ into the neighborhood $\calU$ of $\infty$.
\end{proof}

\subsection{Detecting irreducibility}

\begin{lemma}\label{lem:LocalIrred}
    Let \(h\in k[x]\) be an irreducible polynomial and let \(v\) be a place such that \(h\) has a root \(\theta\in k_v\).  If \(g\in k[x]\) is a polynomial such that \(g - \theta\) is irreducible over \(k_v\), then \(h\circ g\) is irreducible over \(k\).
    Moreover, if $h$ and $g-\theta$ are separable, then so is $h \circ g$.
\end{lemma}

\begin{proof}
Let $\alpha$ be a zero of $g(x)-\theta$ in some extension of $k(\theta)$,
so $\alpha$ is a zero of $h \circ g$.
Since $g(x)-\theta$ is irreducible over $k_v$, it is irreducible over $k(\theta)$,
so $[k(\alpha):k(\theta)]=\deg g$.
On the other hand, $h$ is irreducible, so $[k(\theta):k] = \deg h$.
Multiplying gives $[k(\alpha):k] = \deg(h \circ g)$, so $h \circ g$ is irreducible over $k$.
If, moreover, $h$ and $g-\theta$ are separable, then the same argument with separable degrees shows that $h \circ g$ is separable.
\end{proof}

\section{Brauer group of a rational function field in one variable}
\label{S:Faddeev}

\begin{thm}[Faddeev]\label{thm:Faddeev}
    Let \(k\) be a field, and let \(k^{\sep}\) be a separable closure.  There is an exact sequence
    \[
    0 \To \Br k \To \ker\left(\Br \kk(\PP^1_k)\to\Br \kk(\PP^1_{k^{\sep}})\right)
    \stackrel{\del}\To \!\!\bigoplus_{t\in (\PP^1_k)^{(1)}}\!\!\HH^1(\kk(t), \Q/\Z) \stackrel{\Cor}\To \HH^1(k, \Q/\Z) \To 0
    \]
    in which $\del = (\del_t)$ is given by residue maps, and $\Cor = \sum_{t \in (\PP^1_k)^{(1)}} \Cor_{\kk(t)/k}$.
\end{thm}

\begin{proof}
    A proof for perfect $k$ can be found in~\cite{CTS-BrauerBook}*{Proposition~1.5.1 and Theorem~1.5.2}.  
    The same proof yields the more general statement given here after the minor change of replacing an algebraic closure \(\kbar\) with a separable closure \(k^{\sep}\).  
    One detail: the proof of \cite{CTS-BrauerBook}*{Proposition~1.5.1} shows that
    \(\HH^2(\Gal(k^{\sep}/k), \kk(\PP^1_{k^{\sep}})) = \ker\left(\Br \kk(\PP^1_k)\to\Br \kk(\PP^1_{k^{\sep}})\right)\), and, when $k$ is perfect, simplifies this by using $\Br \kk(\PP^1_{k^{\sep}}) = 0$ (Tsen's theorem); we just skip the last step.
    (When $k$ is imperfect, \(\Br \kk(\PP^1_{k^{\sep}})\) is nontrivial.)
\end{proof}

\begin{remark}
    For \(t\in (\PP^1_k)^{(1)}\), we have \(\HH^1(\kk(t), \Q/\Z) \isom \HH^1(L_t, \Q/\Z)\), where \(L_t\subset \kk(t)\) is the maximal separable subextension.
\end{remark}

\section{Generalities on conic bundles}
\label{S:generalities}

Let $k$ be a field of characteristic not~$2$.

\subsection{Conics}\label{subsec:conics}
A \defi{conic} over $k$ is a closed subscheme of $\PP^2 \colonequals \PP^2_k$ cut out by a nonzero quadratic form.
A conic is \defi{split} if it is isomorphic to $\PP^1$, or, equivalently, is smooth and has a $k$-point.

\begin{lemma}
\label{L:conic split by quadratic extension}
Let $X$ be a conic over $k$.
Let $a \in k^\times$.
Then $X_{k(\sqrt{a})}$ is split if and only if $X$ is isomorphic to $x^2 - a y^2 - b z^2 = 0$ for some $b \in k^\times$. 
\end{lemma}

\begin{proof}
Because of the correspondence between conics and quaternion algebras,
this is equivalent to \cite{CTS-BrauerBook}*{Prop.~1.1.9}. 
\end{proof}

\begin{lemma}[\cite{CTS-BrauerBook}*{Eqn.~(7.3) and Prop.~7.1.3}]
\label{lem:BrConic}
Let $X$ be a smooth conic over $k$.
Let $[X] \in \Br k$ be the class of the associated quaternion algebra.
Then $\Br k \to \Br X$ is surjective, with kernel generated by $[X]$.
\end{lemma}

\subsection{Conic bundles over \texorpdfstring{\(\PP^1\)}{P1}}\label{subsec:Conic bundles}

Given a morphism of varieties $\pi \colon X \to \PP^1 \colonequals \PP^1_k$ and $t \in \PP^1$, let $X_t = \pi^{-1}(t)$.
Let $\eta \in \PP^1$ be the generic point, so $X_\eta$ is the generic fiber.

In this paper, a \defi{conic bundle} is a regular projective geometrically integral surface $X$ with a proper map to $\PP^1$ whose fibers are conics and whose generic fiber is smooth; because of \cite{CTS-BrauerBook}*{Lemma~11.3.2}, there is no harm in assuming also that the fibers are irreducible and reduced (but not necessarily geometrically irreducible).
The degeneracy locus $D$ of $X \to \PP^1$ consists of the $t \in \PP^1$ such that $X_t$ is not smooth (over $\kk(t)$).
A conic bundle is \defi{split} if its generic fiber is split.

\subsection{Conic bundles split by a constant quadratic extension}\label{subsec:QuadraticConic}

The base for our conic bundles will be $\PP^1 \colonequals \Proj k[u_0,u_1]$.
Let $u=u_0/u_1$ and $\Aff^1 \colonequals \Spec k[u] \subset \PP^1$.
There is a bijection between squarefree homogeneous polynomials $F \in k[u_0,u_1]$ of even degree and squarefree polynomials $f \in k[u]$ given by
\[  
    F(u_0,u_1) \mapsto F(u,1) \qquad\text{and}\qquad f(u) \mapsto f\left(\textstyle \frac{u_0}{u_1}\right)u_1^{2 \left\lceil \deg f/2 \right\rceil}.
\]

Let $a \in k^\times$.
Let $f \in k[u]$ be a squarefree polynomial of degree $d \ge 1$.
Let $F \in k[u_0,u_1]$ be the corresponding homogeneous polynomial of even degree $2 \lceil d/2 \rceil$.
The regular affine surface
\[
    y^2 - a z^2 - f(u) = 0 \qquad \textup{in $\Aff^3 = \Spec k[u,y,z] \to \Aff^1 = \Spec k[u]$}
\]
is an open subvariety of the conic bundle $X=X_{a,f} \stackrel{\pi}\to \PP^1$ defined by
\[
   w_0^2 -a w_1^2 - F(u_0,u_1) w_2^2 = 0 \qquad \textup{in $\PP\left(\OO_{\PP^1}\oplus \OO_{\PP^1}\oplus \OO_{\PP^1}\left(\textstyle{\left\lceil\frac{d}2\right\rceil}\right)\right)$.}
\]
Its degeneracy locus $D$ is the zero locus $V(F) \subset \PP^1$.
In particular, $\infty \in D$ if and only if $F(1,0)=0$, which happens if and only if $d$ is odd.
If $d$ is even, then the fiber $X_\infty$ is the smooth conic $w_0^2 - a w_1^2 - b w_2^2$, where $b$ is $F(1,0)$, the leading coefficient of $f$.

A polynomial in $k[u]$ or a homogeneous polynomial $k[u_0,u_1]$ is \defi{separable} if it has no repeated factors over an algebraic closure $\kbar$; this condition is stronger than squarefree if $k$ is imperfect.
If $f$ is separable, then $X_{a,f}$ is smooth over $k$.

\begin{lemma}
\label{L:conic bundle split by quadratic extension}
        Let \(\pi\colon X \to \PP^1\) be a conic bundle over \(k\).  
        Let $a \in k^\times$.
        Then \(X_{k(\sqrt{a})} \to \PP^1_{k(\sqrt{a})}\) is split if and only if \(X\) is birational to \(X_{a,f}\) over $\PP^1$ for some squarefree polynomial $f$.
\end{lemma}

 \begin{proof}
 By Lemma~\ref{L:conic split by quadratic extension} over $k(u)$,
 $X_{k(\sqrt{a})}$ is split if and only if the generic fiber $X_\eta$ of $X \to \PP^1$ is isomorphic to $x^2 - a y^2 - f z^2 = 0$ for some $f \in k(u)^\times$.
 If the latter holds, we can multiply $f$ by a square to assume that $f$ is a squarefree polynomial.
\end{proof}

\subsection{Brauer groups of conic bundles}\label{subsec:BrConic}

    \begin{lemma}\label{lem:TrivialConicBundle}
        Let \(a \in k^{\times 2}\), let \(f\) be a squarefree polynomial, and let \(X \colonequals X_{a,f}\).  
        Then \(\Br X = \Br_0 X = \Br k\).
    \end{lemma}
    
    \begin{proof}
    Write $a=b^2$ with $b \in k^\times$.
        The generic fiber $X_\eta$ of $X \to \PP^1$ is a smooth conic with a rational point $[b:1:0]$, so $X_\eta \isom \PP^1_{k(u)}$.
        Applying \cite{Poonen2017}*{Lemma~6.9.8} twice yields $\Br X \isom \Br \PP^1 \isom \Br k$.
    \end{proof}

\begin{theorem}\label{theorem:BrConicBundle}
        Let \(a\in k^{\times}\).
        Let $f \in k[u]$ be a squarefree polynomial, let \(F \in k[u_0,u_1]\) be the corresponding  homogeneous squarefree polynomial of even degree, and let $\pi\colon X \to \PP^1_k$ be the conic bundle morphism of $X = X_{a,f}$.
    \begin{enumerate}[label*=(\alph*), font=\normalfont]
\item We have $2 \, \overline{\Br}\, X = 0$. \label{part:2torsion}
     
\item  \label{part:BrP1}
        Let $V(F)_{\spl}$ be the set of closed points $t \in V(F)$ that split in $\PP^1_{k(\sqrt{a})} \to \PP^1_k$, i.e., are such that $a \in \kk(t)^{\times 2}$.
        Define
        \begin{align*}
        \calG &\colonequals \{g \in \kk(\PP^1)^\times \mid v_t(g) \equiv 0 \pmod{2} \textup{ for all $t \notin V(F)$}\} \\
        \calG_{\spl} &\colonequals \{g \in \kk(\PP^1)^\times \mid v_t(g) \equiv 0 \pmod{2} \textup{ for all $t \notin V(F)_{\spl}$}\}.
        \end{align*}
        Then
        \[
    1 \To  \calG_{\spl} \To \calG \xrightarrow{\;(a,-)\;} \frac{\Br \kk(\PP^1)}{\Br k}
\]
is exact.

\item \label{part:contained in Br X}
        The image of the composition
            $\calG \xrightarrow{\;(a,-)\;} \Br \kk(\PP^1) \stackrel{\pi^*}\To \Br \kk(X)$
        is contained in $\Br X$, and
\[
    1 \To \langle f, \calG_{\spl} \rangle \To \calG \To \overline{\Br}\, X \To 0
\]
is exact.

\item \label{part:f in Gsplit}
The following are equivalent:
\begin{enumerate}[label=(\roman*), ref=(\roman*), font=\normalfont]
\item \label{I:f} $f \in \calG_{\spl}$;
\item \label{I:V(F)} $V(F)_{\spl} = V(F)$;
\item \label{I:norm} $f$ is an element of $k^\times$ times a norm from the extension $\kk(\PP^1_{k(\sqrt{a})})/\kk(\PP^1)$;
\item \label{I:(a,f)} The element $(a,f) \in \Br \kk(\PP^1)$ lies in $\Br k$;
\item \label{I:del_t} $\del_t (a,f) = 1$ for all $t \in (\PP^1)^{(1)}$;
\item \label{I:square} $a \in \kk(t)^{\times 2}$ for all $t \in V(F)$;
\item \label{I:Gsplit} $\calG_{\spl}=\calG$.
\end{enumerate}

\item \label{part:order of Br-bar X}
Write $F=F_1 \cdots F_n$ with $F_i$ irreducible; suppose that $n \ge 1$.
For each $i$, let $f_i = F_i(u,1)$.
Suppose that for each $i$, the degree of $F_i$ is even and $a$ is not a square in $k[u]/(f_i(u))$.
Then the images of $(a,f_i)$ generate the $\F_2$-vector space $\overline{\Br}\, X$, 
and the only $\F_2$-linear relation they satisfy is that their sum is $0$.
In particular, $\overline{\Br}\, X \isom (\Z/2\Z)^{n-1}$.
\end{enumerate}
    \end{theorem}

    \begin{proof}
    If \(\Char(k) = 0\), this follows from~\cite{CTS-BrauerBook}*{Prop.~11.3.4}.  A similar argument gives the result in positive characteristic, but we do not know a reference, so we provide a proof.
    
        \textbf{\ref{part:2torsion}}         Let \(L:= k(\sqrt{a})\) and \(G:= \Gal(L/k)\).
        Lemma~\ref{lem:TrivialConicBundle} gives the second row of the diagram
        \[
        \begin{tikzcd}
        0
        \arrow[r]
        &
        \Br k 
        \arrow[r] \arrow[d] 
        & 
        \Br X
        \arrow[r] \arrow[d]
        &
        \overline{\Br}\, X
        \arrow[r] \arrow[d] 
        & 
        0
        \\
        0 
        \arrow[r]
        &
        (\Br L)^G
        \arrow[r]
        &
        (\Br X_L)^G 
        \arrow[r]
        & 
        0
        \end{tikzcd}
\]        
The first vertical map is surjective because of the Hochschild--Serre spectral sequence for group cohomology \cite{SerreGaloisCohomology}*{I.\S2.6(b)}.
Therefore the snake lemma produces an exact sequence
\[
   0 \To \ker(\Br k \to \Br L) \To \ker(\Br X \to \Br X_L) \To \overline{\Br}\,X \To 0.
\]
On the other hand, the Hochschild--Serre spectral sequence for \'etale cohomology \cite{Poonen2017}*{Thm.~6.7.5} yields an exact sequence
\[
   0 \To \ker(\Br k \to \Br L) \To \ker(\Br X \to \Br X_L) \To \HH^1(G,\Pic X_L) \To \HH^3(G,L^\times),
\]
 and, since $G$ is cyclic, the last term is isomorphic to $\HH^1(G,L^\times)=0$.
 Comparing the sequences shows that $\overline{\Br}\, X \isom \HH^1(G, \Pic X_L)$, which is killed by $\#G=2$. 

\textbf{\ref{part:BrP1}} For $g \in \calG$, the following are equivalent:
\begin{itemize}
\item $(a,g) \in \Br k$;
\item $\del_t (a,g) = 1$ for all $t \in (\PP^1)^{(1)}$ (by Theorem~\ref{thm:Faddeev});
\item $a^{v_t(g)} \in \kk(t)^{\times 2}$ for all $t \in (\PP^1)^{(1)}$ (since $\del_t (a,g) = a^{v_t(g)}$ in $\kk(t)^\times/\kk(t)^{\times 2}$);
\item $v_t(g) \equiv 0 \pmod{2}$ for all $t \notin V(F)_{\spl}$ (by definition of $V(F)_{\spl}$, since $g\in\calG$); and
\item $g \in \calG_{\spl}$.
\end{itemize}

\textbf{\ref{part:contained in Br X}}
    By Theorem~\ref{thm:Faddeev} (for the top row), \cite{Poonen2017}*{Thm. 6.8.3} (for the bottom row) and the functoriality of pullback (for commutativity), we have the following commutative diagram of exact sequences
        \begin{equation}\label{diag:pullbackBr}
        \begin{tikzcd}
        (\Br k)[2^{\infty}] 
        \arrow[r, hook] \arrow[d] 
        & 
         (\Br \kk(\PP^1))[2^{\infty}] \arrow[d, "\pi^*"] 
         \arrow[r, "\del", two heads] &
         \ker \left(\displaystyle{\bigoplus_{t\in (\PP^1)^{(1)}}} \HH^1(\kk(t), \Q_2/\Z_2) \xrightarrow{\Cor}\HH^1(k, \Q_2/\Z_2)\right)
         \arrow[d, "\Pi"]
        \\
        (\Br X)[2^{\infty}] 
        \arrow[r, hook] & (\Br X_{\eta})[2^{\infty}] 
         \arrow[r] & \displaystyle{\bigoplus_{t\in (\PP^1)^{(1)}}}\,
         \displaystyle{\bigoplus_{\substack{x\in X^{(1)} \\\pi(x)= t} } }  
         \HH^1(\kk(x), \Q_2/\Z_2), 
        \end{tikzcd}
        \end{equation}
        where $\Pi = (\pi_t^*)_{t \in (\PP^1)^{(1)}}$.
        Applying the snake lemma, Lemma~\ref{lem:BrConic}, and \ref{part:2torsion} gives the exact sequence
        \begin{equation} \label{diag:Residues}
           \langle [X_{\eta}]\rangle \stackrel{\del}{\To} \ker\Pi \To \overline{\Br}\, X \To 0.
        \end{equation}

        Let us compute $\ker \Pi$.
        Since the fiber \(X_t\) is geometrically reducible if and only if \(t\in V(F)\), 
        \[
                \ker \Bigl( \pi_t^*\colon \HH^1(\kk(t), \Q_2/\Z_2) \to\displaystyle{\bigoplus_{\substack{x\in X^{(1)} \\\pi(x)= t} }}          \HH^1(\kk(x), \Q_2/\Z_2) \Bigr) = \begin{cases}
             1, & \textup{ if $t\notin V(F)$;}\\
             \langle a \rangle & \textup{ if $t\in V(F)$.}
         \end{cases}
        \]
        Thus, \(\ker \Pi\) is the intersection of \(\ker (\pi_t^*)_{t\in (\PP^1)^{(1)}} = \bigoplus_{t\in V(F)}\langle a\rangle \oplus \bigoplus_{t\notin V(F)} \{1\}\) with \(\ker(\Cor)\).  Note that for \((a^{c_t})_{t\in V(F)}\), 
        \[
        \Cor((a^{c_t})_{t\in V(F)}) = \prod_{t\in V(F)}a^{c_t\deg(t)} = a^{\sum c_t\deg(t)}.
        \]  
        Therefore,
        \[
            \ker \Pi = \Bigl\{ (a^{c_t})_{t\in V(F)} \in  \bigoplus_{t\in V(F)}\langle a\rangle \oplus \bigoplus_{t\notin V(F)} \{1\} : \sum_{t\in V(F)} c_t\deg(t)\equiv 0 \pmod 2 \Bigr\}
        \]
        
If $g \in \calG$, then $\del_t(a,g) = a^{v_t(g)} \in \ker \pi_t^*$, so $\calG$ maps into $\ker \Pi$.  

We now prove that \(\calG \to \ker \Pi\) is surjective.
Suppose that $(a^{c_t})_{t \in V(F)}$ is an element of $\ker \Pi$, where $c_t \in \Z$ and $\sum_{t\in V(F)} c_t\deg(t) =2r$ for some $r \in \Z$.
Then $ \sum_{t\in V(F)} c_t\cdot [t] - 2r [\infty] = \divv(g)\) for some function $g \in \kk(\PP^1)^\times$.  
Now $g \in \calG$, and $\del_t(g) = a^{v_t(g)}$, which is $a^{c_t}$ if $t\in V(F)$ 
and $1$ if $t\notin V(F)$.

We have injections $\calG/\calG_{\spl} \xrightarrow{\; (a,-) \;} \frac{\Br \kk(\PP^1)}{\Br k} \stackrel{\del}\To \ker(\Cor)$, by \ref{part:BrP1} and Theorem~\ref{thm:Faddeev}, respectively.
The composition defines an isomorphism $\calG/\calG_{\spl} \to \ker \Pi$, by the previous paragraph.
This isomorphism maps $f$ to $\del [X_\eta]$,
so $\calG/\langle f, \calG_{\spl} \rangle \isom (\ker \Pi)/\langle \del [X_\eta] \rangle \isom \overline{\Br}\, X$, by \eqref{diag:Residues}.

\textbf{\ref{part:f in Gsplit}}
The following are immediate:
\ref{I:f}$\Rightarrow$\ref{I:V(F)}$\Rightarrow$\ref{I:Gsplit}$\Rightarrow$\ref{I:f} and
\ref{I:(a,f)}$\Rightarrow$\ref{I:del_t}$\Rightarrow$\ref{I:square}$\Rightarrow$\ref{I:V(F)}.
Moreover, \ref{I:V(F)}$\Rightarrow$\ref{I:norm} since if $V(F)_{\spl}=V(F)$, then the \emph{divisor} of $f$ is a norm.
Finally, \ref{I:norm}$\Rightarrow$\ref{I:(a,f)} since $(a,g)=0$ in $\Br \kk(\PP^1)$ whenever $g \in \kk(\PP^1)^\times$ is a norm.

\textbf{\ref{part:order of Br-bar X}}
The assumption on $a$ implies that $V(F)_{\spl}=\emptyset$,
so $f_1,\ldots,f_n$ map to an $\F_2$-basis of $\calG/\calG_{\spl}$.
Also, $f=f_1\cdots f_n$.
Thus \ref{part:order of Br-bar X} follows from \ref{part:contained in Br X}.
\qedhere
    \end{proof}

\section{Conic bundles with prescribed Brauer--Manin pairing}
\label{S:base change}
    
In this section, we prove Theorem~\ref{thm:Main} by using the following theorem to construct a conic bundle over \(\PP^1\).

\begin{thm}\label{thm:MainBaseChangeThm}
    Let \(k\) be a global field of characteristic not \(2\), let $n,r \ge 0$, and let \(S_1, \ldots, S_r\subset \widehat{(\Z/2\Z)^n}\).
    Then there exist a conic bundle \(X\to \PP^1_k\) that is smooth over $k$ and split by a quadratic extension of~$k$, an injection \( (\Z/2\Z)^n \injects \Br X \), and places \(v_1, \ldots, v_r\) of \(k\) such that
    \begin{enumerate}[label*=(\roman*), font=\normalfont]
        \item \label{part:BaseChangeBr}
            The composition \( (\Z/2\Z)^n \injects \Br X \surjects \overline{\Br}\, X\) is an isomorphism;
        \item \label{part:S_i}
        For $j=1,\ldots,r$, the image of \(X(k_{v_j}) \to \widehat{(\Z/2\Z)^n}\) is  \(S_j\); and
        \item \label{part:constant}
        For all places \(w \notin \{v_1,\ldots,v_r\}\), the image of \(X(k_w)\to \widehat{(\Z/2\Z)^n}\) is $\{0\}$. 
    \end{enumerate}
\end{thm}

\begin{proof}[Proof of Theorem~\ref{thm:Main}]
    Apply Theorem~\ref{thm:MainBaseChangeThm} with \(r \colonequals 1\), $n \colonequals N+1$, and \hbox{\(S_1 \colonequals  \widehat{(\Z/2\Z)^n}\setminus\{0\}\)} to produce a conic bundle \(X\) with \(S = \hat{B} \setminus \{0\}\), where $B \colonequals \im((\Z/2\Z)^n \to \Br X)$.  Corollary~\ref{cor:2torsion} shows that no proper subgroup of \(B\) gives an obstruction.  Since \(B\) surjects onto \(\overline{\Br}\, X\), this Brauer--Manin obstruction is not captured by any subgroup generated by $N$ elements.
\end{proof}

\begin{remark}
    One can also take \(r>1\), provided there are sets \(S_1,\ldots,S_r\) with $\sum S_i = \widehat{(\Z/2\Z)^n}\setminus\{0\}$.  
    If $1 \le r \le n-1$, then such sets $S_1,\ldots,S_r$ of size $\ge 2$ exist, 
    by a minor variation of Theorem~\ref{thm:sharpboundnew}\ref{I:example of sum S_i}. 
\end{remark}

In Section~\ref{sec:StartingConicBundle}, we produce a conic bundle \(\tilde{\pi}\colon\tilde{X}\to\PP^1\) with an injection $(\Z/2\Z)^n \to \Br X$ such that the local evaluation map $\tilde{X}(k_v) \to \widehat{(\Z/2\Z)^n}$ is surjective for several places $v$.  Finally, in Section~\ref{sec:ProofOfMainBaseChange}, we prove Theorem~\ref{thm:MainBaseChangeThm} by constructing a map \(g \colon \PP^1\to \PP^1\) such that the base change of  \(\tilde{X} \to \PP^1\) by $g$ has the desired properties.

\subsection{Conic bundles with surjective evaluation at arbitrarily many places}\label{sec:StartingConicBundle}

\begin{proposition}\label{P:surjective}\label{prop:StartingConicBundle}
Given $n,r \ge 0$, there exist 
a conic bundle \(\tilde{\pi} \colon \tilde{X}\to \PP^1\) that is smooth over $k$ and split by a quadratic extension of $k$,  
an injection $(\Z/2\Z)^n \injects \Br \tilde{X}$,
and nonarchimedean places \(\tilde{v}_1, \ldots, \tilde{v}_r\) of \(k\) not lying over \(2\)
such that 
\begin{enumerate}[font=\normalfont]
    \item \label{part:IndepBrauerClasses}
    The composition \( (\Z/2\Z)^n \injects \Br \tilde{X} \surjects \overline{\Br}\, \tilde{X}\) is an isomorphism; 
    \item Every point in the degeneracy locus of \(\tilde{\pi}\) has even degree;\label{part:EvenDegree}
    \item For all archimedean places \(w\), the conic bundle \(\tilde{X}_{k_w}\) is birational to a \(\PP^1\)-bundle;\label{part:archimedean}
    \item For $j=1,\ldots,r$, the map \(\phi_{\tilde{v}_j}\colon \tilde{X}(k_{\tilde{v}_j}) \to \widehat{(\Z/2\Z)^n}\) is surjective; and\label{part:Surjective}
    \item For all places \(w\), we have \(0\in \im \phi_w\).\label{part:ImContains0}
\end{enumerate}
\end{proposition}

\begin{proof}
Choose distinct monic separable irreducible polynomials \(\tilde{f}_0,\ldots,\tilde{f}_{n}\) of even degree.
Let \(\tilde{f} = \prod \tilde{f}_i\).
Let \(\tilde{v}_1, \ldots, \tilde{v}_r\) be nonarchimedean places not dividing \(2\) such that \(\tilde{f}_i \in \calO_{\tilde{v}_j}[u]\) for all \(i\) and \(j\), the reduction \(\tilde{f} \bmod \tilde{v}_j\) is separable for all \(j\), and each residue field \(\F_{\tilde{v}_j}\) is large enough that the set considered in Proposition~\ref{P:Weil bound} is nonempty for all choices of $\varepsilon_j$.
Choose \(a \in k^\times\) such that \(a\) is totally positive, \(\tilde{v}_j(a)=1\) for all \(j\), and $k(\sqrt{a})$ does not embed in $k[u]/(\tilde{f}_i)$ for any~$i$.  
Let \(\tilde{X} \colonequals X_{a,\tilde{f}}\).
Let $(\Z/2\Z)^n \to \Br \tilde{X}$ be the homomorphism sending the $i$th standard generator to $(a,\tilde{f}_i)$ for $i =1,\ldots,r$ (not $0$).

We claim that these have the desired properties.  
Since $\tilde{f}$ is monic, \(\tilde{X}\) has a \(k\)-point $P$ at infinity.
Then $P^* (a,\tilde{f}_i) = (a,1) = 0$ for all $i$, so $\phi_w(P)=0$ for all $w$, so \eqref{part:ImContains0} holds.
Since each \(\tilde{f}_i\) is irreducible of even degree,~\eqref{part:EvenDegree} holds.  Since $a$ is not a square in $k[u]/(\tilde{f}_i)$ for any $i$, Theorem~\ref{theorem:BrConicBundle}\ref{part:order of Br-bar X} implies~\eqref{part:IndepBrauerClasses}.  The assumption that \(a\) is totally positive implies~\eqref{part:archimedean}.

It remains to show~\eqref{part:Surjective}, that \(\phi_{\tilde{v}_j}\) is surjective for each \(i\).  Fix \(w \colonequals \tilde{v}_j\), to simplify notation.  
Suppose that $\varepsilon_0,\ldots,\varepsilon_n \in \F_w^\times$ have product $1$.
Proposition~\ref{P:Weil bound} produces $c \in \F_w$ such that $\tilde{f}_i(c)\varepsilon_i\in \F_w^{\times 2}$ for all $i$.
Multiplying gives $\tilde{f}(c)=1$.
Lift $c$ to some $c' \in \calO_w$.
By Hensel's lemma, $\tilde{f}(c') \in \calO_w^{\times 2}$, so there exists $P \in \tilde{X}(k_w)$ with $u$-coordinate $c'$.
Then $\phi_w(P)$ maps $(a,\tilde{f}_i)$ to $\inv_w (a,\tilde{f}_i(c'))$, which is $0$ or $1/2$ according to whether $\varepsilon_i \in \F_w^{\times 2}$ or not, since $w(a)=1$.
Since $\varepsilon_1,\ldots,\varepsilon_n$ are arbitrary (and then $\varepsilon_0$ is determined), $\phi_w(P)$ can be made to equal any homomorphism $(\Z/2\Z)^n \to \Q/\Z$.
\end{proof}

\subsection{Proof of Theorem~\ref{thm:MainBaseChangeThm}}\label{sec:ProofOfMainBaseChange}
Given our fixed \(n, r\), we apply Proposition~\ref{prop:StartingConicBundle} to obtain a conic bundle \(\tilde{\pi}\colon\tilde{X} = X_{a,\tilde{f}} \to \PP^1\) split by a quadratic extension of $k$, an injection \((\Z/2\Z)^n \injects \Br X\), and places \(\tilde{v}_1, \dots, \tilde{v}_r\) such that each \(\phi_{\tilde{v}_j}\) is surjective.
There are finitely many additional places $w$ for which $\phi_w$ is nonconstant; call them \(\tilde{v}_{r+1}, \dots, \tilde{v}_m \); they are nonarchimedean because of Proposition~\ref{prop:StartingConicBundle}\eqref{part:archimedean}.
In proving Theorem~\ref{thm:MainBaseChangeThm}, we are free to enlarge $r$ to $m$ and set the new $S_j$ equal to $\{0\}$; now the \(\phi_{\tilde{v}_j}\) are no longer all surjective, but $S_j \subset \im \phi_{\tilde{v}_j}$ still holds for all $j$, by Proposition~\ref{prop:StartingConicBundle}\eqref{part:ImContains0}.
In addition, for any place $w \notin \{\tilde{v}_1,\ldots,\tilde{v}_r\}$, the constant map $\phi_w$ has image $\{0\}$ by Proposition~\ref{prop:StartingConicBundle}\eqref{part:ImContains0}.
Let $v_j=\tilde{v}_j$ for each $j$.
Let $D$ be the degeneracy locus of $\tilde{\pi}$, which consists of closed points $D_i$ of even degree, the zero loci in $\Aff^1$ of the irreducible factors $\tilde{f}_i$ of $\tilde{f}$.

For $j=1,\ldots,r$, let $\calU_{v_j} = \tilde{\pi}(\phi_{v_j}^{-1}(S_j))$, which is the union over $s \in S_j$ of the nonempty open sets $\calU_{v_j,s} \colonequals \tilde{\pi}(\phi_{v_j}^{-1}(s))$ in $\PP^1(k_{v_j})$.
We can arrange that $\infty \in \calU_{v_j}$ for all $j$:
use weak approximation to pick $c \in (\PP^1 \setminus D)(k)$ such that $c \in \calU_{v_j}$ for every $j$,
and change coordinate on $\PP^1$ to assume $c=\infty$.
Then the fiber $\tilde{X}_\infty$ is a smooth curve and $\tilde{X}_\infty(k_{v_j}) \neq \emptyset$ for all~\(j\).
For each $i$, let $\theta_i$ be a zero of $\tilde{f}_i$ in some finite extension of $k$ and choose a distinct nonarchimedean place $w_i \notin \{v_1,\ldots,v_r\}$ such that $\tilde{X}_{\infty}(k_{w_i})\neq\emptyset$ and $w_i$ splits completely in $k(\sqrt{a},\theta_i)$; then $\sqrt{a}, \theta_i \in k_{w_i}$.

Consider the following conditions on a polynomial $g$ over $k$: 
\begin{enumerate}
    \item \label{part:w_j}
        For each $i$, the polynomial \(g - \theta_i\) is irreducible and separable over \(k_{w_i}\);
    \item \label{part:v_i}
        For each $j$, we have \(g(\PP^1(k_{v_j}))\subset \calU_{v_j}\) and \(g(\PP^1(k_{v_j}))\) meets $\calU_{v_j,t}$ for each $t \in S_j$; and 
    \item \label{part:w}
        For each place \(w\) such that \(\tilde{X}_{\infty}(k_w) =\emptyset\), the set $g(\PP^1(k_w))$ meets $\tilde{\pi}(\tilde{X}(k_w))$.
\end{enumerate}
These are open conditions at finitely many different places, each satisfiable in sufficiently divisible degrees (\eqref{part:v_i} by Lemma~\ref{L:special polynomial}), so by weak approximation there is a global polynomial $g$ satisfying them all.

Let $\pi \colon X = X_{a,\tilde{f} \circ g} \to \PP^1$ be the pullback of $\tilde{\pi} \colon \tilde{X} \to \PP^1$ along $g$.
The degeneracy locus of $\pi$ is $g^{-1}(D) = \Union g^{-1}(D_i)$, and each $g^{-1}(D_i)$ is the zero locus of $\tilde{f}_i \circ g$.
Condition~\eqref{part:w_j} and Lemma~\ref{lem:LocalIrred} imply that each $\tilde{f}_i \circ g$ is irreducible and separable, so $X$ is smooth over $k$.
Moreover, Lemma~\ref{lem:LocalIrred} over $k(\sqrt{a})$ implies that $\tilde{f}_i \circ g$ remains irreducible over $k(\sqrt{a})$, so $a$ is not a square in $k[u]/(\tilde{f}_i \circ g)$. 
Thus, by Theorem~\ref{theorem:BrConicBundle}\ref{part:order of Br-bar X}, 
the images of $(a,\tilde{f}_i \circ g)$ in $\overline{\Br}\, X$ for $i=0,\ldots,n$ generate $\overline{\Br}\, X$,
and the only $\F_2$-linear relation between them is that their sum is $0$.
In other words, omitting the $0$th generator, we find that the composition $(\Z/2\Z)^n \to \Br \tilde{X} \to \Br X \to \overline{\Br}\, X$ is an isomorphism; this is part~\ref{part:BaseChangeBr} of the theorem.

Parts \ref{part:S_i} and~\ref{part:constant} of the theorem concern the composition
\[
    X(k_w) \To \tilde{X}(k_w) \stackrel{\phi_w}\To \widehat{(\Z/2\Z)^n}
\]
for places $w$.
Condition~\ref{part:v_i} on \(g\) implies~\ref{part:S_i}. 
For $w \notin \{v_1,\ldots,v_r\}$, we already noted that \(\phi_w\) is $0$,
and condition~\ref{part:w} on \(g\) implies $X(k_w) \ne \emptyset$,
so \ref{part:constant} holds.
\qed

\section{An explicit construction over \texorpdfstring{\(\Q\)}{Q}} 
\label{sec:explicit}

In this section we give an explicit construction of a conic bundle over \(\Q\) with the desired properties of Theorem~\ref{thm:Main}.  This example is constructed using a base change argument reminiscent of the proof of Theorem~\ref{thm:MainBaseChangeThm}; see Remark~\ref{rem:base_change_comparison} for a comparison between these two approaches.

Let \(n \geq 2\) be a positive integer and let \(q\) be a prime number that is larger than \(n\).
By Dirichlet's theorem on primes in arithmetic progressions, there is a prime \(p\) that is congruent to \(1\) modulo \(8\), congruent to a non-square modulo \(q\), congruent to a square for all odd primes \(\ell\leq (q-1)(n-1)\) different from \(q\), and large enough that the conclusion of Proposition~\ref{P:Weil bound} holds when applied to polynomials of degree \(n+1\).

Define \(\tilde{f}_{0}(u) \colonequals q u + 4n\).  For \(i= 1, \ldots, n\), define \(\tilde{f}_i(u) \colonequals u + 4(n-i)\).  
By Proposition~\ref{P:Weil bound}, the image of $(\tilde{f}_0,\ldots,\tilde{f}_n) \colon \F_p \to \F_p^{n+1}$ meets every coset of $(\F_p^{\times 2})^{n+1}$ in $(\F_p^\times)^{n+1}$.
Let \(\mathcal{E}\) be the set of nonidentity elements $(\eps_0,\ldots,\eps_n) \in (\F_p^{\times}/\F_p^{\times2})^{n+1}\) with \(\prod \eps_i = 1\).
Choose a function $\psi \colon \A^1(\F_p) \to \F_p$ such that the composition
    \begin{equation}
    \label{equation:defined and image E}
        \A^1(\F_p) \stackrel{\psi}\To \F_p \xrightarrow{(\tilde{f}_0,\ldots,\tilde{f}_n)} (\F_p)^{n+1} \dashrightarrow (\F_p^\times/\F_p^{\times 2})^{n+1}
    \end{equation}
is defined and has image $\calE$.
By Lemma~\ref{lem:PolynomialInterpolation}, there exists a monic polynomial \(h \in \F_p[u]\) of degree \(p+1\) that on $\A^1(\F_p)$ agrees with $\psi$.
Fix a monic integral polynomial \(g \in \Z[u]\) of degree \(p+1\) such that
\[
    g\equiv h \pmod p, \quad g \equiv u^{p+1} - u^{p+2 -q} + 4\pmod{q}, \quad\textup{and}\quad g\textup{ is Eisenstein at }2.
\]
Let \(f_i \colonequals \tilde{f}_i \circ g\) and \(f = \prod_{i=0}^{n} f_i\); these are in $\Z[u]$.

\begin{lemma} \label{lem: propertiesoff} Let $p$, $q$, $n$, $f_i$, and $f$ be defined as above.
    \begin{enumerate}[font=\normalfont]
        \item\label{part:irred} 
        Each \(f_i\) is irreducible of even degree.  Its leading coefficient is $q$ if $i=0$, and $1$ if $i>0$.
        \item\label{part:q} Each $f_i$ maps $\A^1(\Z_q)$ into $\Z_q^\times$.
        \item\label{part:p} Each $f_i$ maps $\A^1(\Z_p)$ into $\Z_p^\times$, and the composition
        \begin{equation}\label{equation:image E}
           \A^1(\Z_p) \xrightarrow{(f_0,\ldots,f_n)} (\Z_p^\times)^{n+1} \surjects (\F_p^\times/\F_p^{\times 2})^{n+1}
        \end{equation}
        has image $\calE$.
        \item\label{part:f(c) for c in Z_p}
        The polynomial $f$ maps $\A^1(\Z_p)$ into $\Z_p^{\times 2}$.
        \item\label{part:ell} If \(\ell > (q-1)(n-1)\) is prime and \(c \in \A^1(\Z_\ell)\), then at most one of $f_0(c),\ldots,f_n(c)$ is not an \(\ell\)-adic unit.
    \end{enumerate}
\end{lemma}

\begin{proof} \hfill

  \textbf{\eqref{part:irred}:} Each $f_i$ has degree $p+1$.
  Since $g$ is Eisenstein at $2$, it is irreducible, so $f_i$ is irreducible too.
  Since $g$ is monic, the leading coefficient of $f_i \colonequals \tilde{f}_i \circ g$ equals that of $\tilde{f}_i$.
  
   \textbf{\eqref{part:q}:} For any \(c \in \A^1(\Z_{q})\), we have \(g(c)\equiv 4\pmod q\), so 
   $f_0(c) \equiv \tilde{f}_0(4) \equiv 4n \not\equiv 0 \pmod{q}$
   and $f_i(c) \equiv \tilde{f}_i(4) \equiv 4(n-i+1) \not\equiv 0\pmod{q}$ for $i=1,\ldots,n$,
   since \(q > n\) by assumption.
   
    \textbf{\eqref{part:p}:} 
    The composition is same as the composition
    \[
    \A^1(\Z_p) \surjects \A^1(\F_p) \xrightarrow{g=h=\psi} \F_p \xrightarrow{(\tilde{f}_0,\ldots,\tilde{f}_n)} (\F_p)^{n+1} \dashrightarrow (\F_p^\times/\F_p^{\times 2})^{n+1},
    \]
    which is a surjection followed by \ref{equation:defined and image E}, so its image is $\calE$.

    \textbf{\eqref{part:f(c) for c in Z_p}:} 
    The composition
    \[
        \A^1(\Z_p) \stackrel{f}\To \Z_p^\times \surjects \F_p^\times/\F_p^{\times 2}
    \]
    equals the composition of \ref{equation:image E} with the product map to $\F_p^\times/\F_p^{\times 2}$,
    which sends $\calE$ to $1$.
    Thus $f$ maps $\A^1(\Z_p)$ to $\ker(\Z_p^\times \surjects \F_p^\times/\F_p^{\times 2})$, which equals $\Z_p^{\times 2}$ by Hensel's lemma.
    
    \textbf{\eqref{part:ell}:} Suppose that there exist $i < j$ such that $f_i(c)$ and $f_j(c)$ are not $\ell$-adic units. Let \(c'\colonequals g(c)\); then $\tilde{f}_i(c') \equiv \tilde{f}_j(c') \equiv 0 \pmod{\ell}$.
    If \(i=0\), then the congruences 
    \begin{align*}
    qc'+4n \equiv\tilde{f}_0(c') &\equiv 0 \pmod{\ell}\\
    c' + 4(n-j) \equiv \tilde{f}_j(c') &\equiv 0 \pmod{\ell}
    \end{align*}
    force $q(n-j) \equiv n \pmod{\ell}$; also, $0 < |q(n-j) - n| \leq (q-1)(n-1)$, so $\ell \leq (q-1)(n-1)$.  
    If $i\ne 0$, then the congruence 
    \[
    c' + 4(n-i)\equiv\tilde{f}_i(c') \equiv 0\equiv \tilde{f}_j(c') \equiv c' + 4(n-j)\pmod{\ell}
    \]
    implies $i \equiv j \pmod{\ell}$; also, $0 < |i-j| < n$, so $\ell < n \le (q-1)(n-1)$.
    \end{proof}

The remainder of this section is devoted to proving that \(X = X_{p, f}\) satisfies the conditions of Theorem~\ref{thm:Main}.
\begin{lemma}\label{lem:BrExplicit}
    The Brauer classes $\alpha_i \colonequals (p, f_i)$ for $i=1,\ldots,n$ generate an elementary abelian \(2\)-subgroup of \(\Br X\) that has order \(2^n\) and generates \(\overline{\Br}\, X\).
\end{lemma}
\begin{proof}
     Since \(2\) is unramified in \(\Q(\sqrt{p})\), while \(f_i\) is Eisenstein at $2$, 
     the field $\Q(\sqrt{p})$ does not embed in $\Q[u]/(f_i(u))$.
     The result now follows from Theorem~\ref{theorem:BrConicBundle}\ref{part:order of Br-bar X}.
\end{proof}

\begin{thm} \label{thm: explicitconstruction}
    Consider the conic bundle \(X = X_{p, f}\) and let \(B \colonequals \langle \alpha_1,\ldots,\alpha_n \rangle \subset \Br X\) be the subgroup generated by the classes $\alpha_i$ described in Lemma~\ref{lem:BrExplicit}.  For each place \(v\), let \(S_v\) denote the image of \(\phi_v\colon X(\Q_v) \to \hat{B}\). Then 
    \begin{enumerate}[font=\normalfont]
        \item \label{part:LocalSolubility} 
            \(X(\A_{\Q})\neq\emptyset\),
        \item \label{part:ConstantEvaluation}
            For all places \(v\neq p\), we have \(S_v = \{0\}\), and
        \item \label{part:EvaluationAtp}
            \(S_p = \hat{B} \setminus \{0\}\).
    \end{enumerate}
    In particular, \(X(\A_{\Q})^{\Br} = \emptyset\) and for all subgroups \(B'\subsetneq \Br X\) that do not generate \(\overline{\Br}\, X\), \(X(\A_{\Q})^{B'}\neq\emptyset\).
\end{thm}

\begin{proof}\hfill

    \textbf{\eqref{part:LocalSolubility}:} By construction, \(f\) has even degree and has leading coefficient \(q\).  Thus the fiber $X_\infty$ is the conic \(y^2 - pz^2 = qw^2\).  Since \(p\equiv 1 \pmod 8\) and $p \bmod q$ is not a square, $X_\infty$ has a $\Q_v$-point if and only if $v \ne p,q$.

    Let $c \in \A^1(\Z_p)$; by Lemma~\ref{lem: propertiesoff}\eqref{part:f(c) for c in Z_p}, $f(c) \in \Z_p^{\times 2}$, so the fiber $X_c$ has a $\Q_p$-point.
    
    Let $c' \in \A^1(\Z_q)$; by Lemma~\ref{lem: propertiesoff}\eqref{part:q}, $f(c') \in \Z_q^\times$, so $X_{c'}$ has good reduction at $q$, so $X_{c'}$ has a $\Q_q$-point.

    \textbf{\eqref{part:ConstantEvaluation}:} 
    By \eqref{part:LocalSolubility}, each $S_v$ is nonempty.
    If $p \in \Q_v^{\times 2}$, then all $\alpha_i$ are $0$, so $S_v \subset \{0\}$, so $S_v=\{0\}$.
    Therefore it remains to prove $S_v \subset \{0\}$ for $v=\ell \ne p$, where $\ell=q$ or $\ell > (q-1)(n-1)$.
    That is, for such $\ell$, and for any $i>0$ and any $P \in X(\Q_\ell)$, we need $P^* \alpha_i = 0$ in $\Br \Q_\ell$.
    Let $c=\pi(P)$; if $c \ne \infty$, the condition can be written as $(p,f_i(c)) = 0$ in $\Br \Q_\ell$.

    \emph{Case 1: $c=\infty$.} By the proof of \eqref{part:LocalSolubility}, $\ell \ne p,q$.
By Lemma~\ref{lem: propertiesoff}\eqref{part:irred}, there is an integer \(n\) such that $f_i/t^{2n}$ evaluates to \(1\) at \(\infty\), so $P^* \alpha = 0$ for any $P \in X_\infty(\Q_\ell)$.

    \emph{Case 2: $c \in \A^1(\Q_\ell) \setminus \A^1(\Z_\ell)$.}
    By Lemma~\ref{lem: propertiesoff}\eqref{part:irred}, $v_\ell(f_i(c))$ is even.
    Since $\ell \ne 2,p$, we have $(p,f_i(c))=0$ in $\Br \Q_\ell$.

    \emph{Case 3: $c \in \A^1(\Z_\ell)$.}
    Since $f$ is in the image of the norm from $\Q(X)(\sqrt{p})/\Q(X)$, we have $(p,f)=0$ in $\Br X$.
    Thus $\sum_{i=0}^n (p,f_i(c))=0$.
    If $\ell=q$, Lemma~\ref{lem: propertiesoff}\eqref{part:q} implies $f_i(c) \in \Z_q^\times$, so $(p,f_i(c))=0$ for all $i$.
    If $\ell \ne p,q$ with $\ell>(q-1)(n-1)$, then by Lemma~\ref{lem: propertiesoff}(\ref{part:ell}), there exists $j$ such that for all $i \ne j$, we have $f_i(c) \in \Z_\ell^\times$, so $(p,f_i(c))=0$; then $\sum_{i=0}^n (p,f_i(c))=0$ forces $(p,f_j(c))=0$ too.

    \textbf{\eqref{part:EvaluationAtp}:}    
    We claim that \(\pi(X(\Q_p)) = \A^1(\Z_p)\). 
    By the proof of \eqref{part:LocalSolubility}, $X_\infty(\Q_p)=\emptyset$; that is, $\infty \notin \pi(X(\Q_p))$.
    For $c \in \A^1(\Q_p) \setminus \A^1(\Z_p)$, Lemma~\ref{lem: propertiesoff}\eqref{part:irred} implies $f(c) \in q \Q_p^{\times 2}$, and $\left( \frac{q}{p} \right) = \left( \frac{p}{q} \right) = -1$, so the conic $X_c \colon y^2-pz^2 = f(c) w^2$ has no $\Q_p$-point.
    For $c \in \A^1(\Z_p)$, the proof of \eqref{part:LocalSolubility} showed that $X_c$ has a $\Q_p$-point.
    This proves the claim.

   Since $(p,-) \colon \Z_p^\times \to \Br \Q_p = \Q/\Z$ equals the composition $\Z_p^\times  \surjects \F_p^\times/\F_p^{\times 2} \isom \frac{1}{2}\Z/\Z \injects \Q/\Z$,  the map $\phi_v$ equals the composition
   \[
      X(\Q_p) \stackrel{\pi}\surjects \A^1(\Z_p) \stackrel{\textup{\eqref{equation:image E}}}\To (\F_p^\times/\F_p^{\times 2})^{n+1} \xrightarrow{\textup{forget $0$th coordinate}} (\F_p^\times/\F_p^{\times 2})^n \isom (\tfrac{1}{2}\Z/\Z)^n \isom \hat{B}
   \]
   sending $P$ to $(P^* (p,f_i))_{1 \le i \le n} \in (\tfrac{1}{2}\Z/\Z)^n$ and then applying the isomorphism to $\hat{B}$ defined by our $\F_2$-basis of $B$.
   By Lemma~\ref{lem: propertiesoff}\eqref{part:p}, the image of $X(\Q_p)$ in $(\F_p^\times/\F_p^{\times 2})^{n+1}$ is $\calE$, which maps onto $(\F_p^\times/\F_p^{\times 2})^n \setminus \{(1,\ldots,1)\}$ and then $\hat{B} \setminus \{0\}$.
\end{proof}

\begin{remark} To construct an explicit $X_{p,f}$ as in Theorem~\ref{thm: explicitconstruction}, we must first find suitable primes $p$ and $q$. When $n=4$ we can take $q = 5$, in which case the smallest prime $p$ satisfying the conditions is $p = 1873$, and we may choose $\psi$ to have image
\[
        \{  3, 6, 7, 11, 15, 20, 22, 26, 29, 31, 33, 35, 41, 61, 195 \} \subset \F_{1873}. 
\]
All computations were done in \texttt{Magma}. 
\end{remark}

\begin{remark}\label{rem:base_change_comparison}
    
There is a strong analogy between the construction of the conic bundle \(X_{p, f}\) satisfying Theorem~\ref{thm: explicitconstruction} and the base change technique used in the proof of Theorem~\ref{thm:MainBaseChangeThm}.  Recall the notations of \(\tilde{f}_i, g, \psi, f\).  Let \(\tilde{f} = \prod_i \tilde{f}_i\).  Then the conic bundle \(X \colonequals X_{p, f}\) is the base change of \(\tilde{X} \colonequals X_{p, \tilde{f}}\) by the map \(g \colon \pp^1 \to \pp^1\).  Since the \(\tilde{f}_i\) have odd degree, the classes \(\tilde{\alpha}_i \colonequals (p, \tilde{f}_i)\), whose pullbacks \(g^*\tilde{\alpha}_i = (p, f_i)\) generate $\overline{\Br}\, X$, are ramified along \(\tilde{X}_\infty\). Moreover, using Theorem~\ref{theorem:BrConicBundle}, one sees that the 2-rank of \(\overline{\Br}\, X\) is larger than the 2-rank of \(\overline{\Br}\, \tilde{X}\), so we cannot hope to employ \emph{exactly} the same strategy as in the proof of Theorem~\ref{thm:MainBaseChangeThm}. Nevertheless, \(\tilde{\alpha}_i \in \Br(\tilde{X}_{\A^1})\) and we can evaluate \(\tilde{\alpha}_i\) at points in \(\tilde{X}_{\A^1}(\Z_\ell)\).  (This suffices since \(g\) is totally ramified of even degree over infinity.)  For primes \(\ell \neq p\) or \( q\), the same proof as in Theorem~\ref{thm: explicitconstruction}\eqref{part:ConstantEvaluation} shows that the image of the map \(\tilde{X}_{\A^1}(\Z_\ell) \to \widehat{(\Z/2\Z)^n} \) is \(\{0\}\). By our choice of \(p\) sufficiently large to satisfy Proposition~\ref{P:Weil bound}, the map \(\tilde{X}_{\A^1}(\Z_p) \to \widehat{(\Z/2\Z)^n} \) is surjective.  For \(q\), any $P \in \tilde{X}_{\A^1}(\Z_{q})$ with $\pi(P) \equiv 4 \pmod{q}$ 
is sent to 0 in \( \widehat{(\Z/2\Z)^n} \).  
The constraints on \(g\) modulo \(p\) and \(q\) amount to constraining the image of \(X_{\A^1}(\Z_{\ell})\) in \(\tilde{X}_{\A^1}(\Z_\ell)\) for \(\ell = p, q\) to force \(S_p = \widehat{(\Z/2\Z)^n} \setminus \{0\}\) and \(S_{q} = \{0\}\). 
\end{remark}
\begin{remark}
It is possible to use the ideas of this section to realize any combinatorial assignment as in Section \ref{sec:group}, and for a general global field of characteristic not $2$, not just for $\Q$. Since this is already achieved by the proof in Section \ref{S:base change}, we omit the details.     
\end{remark}

\appendix

\section{Sumsets excluding exactly one element}\label{sec:bounds}    

In this section we use a result of Rado to prove Theorem~\ref{thm:sharpboundnew}, which constrains the ways that a sumset can equal $\F_2^n \setminus \{v\}$ for a vector $v$.

    \begin{lemma}[\cite{Rado1942}*{Theorem 1}] \label{lem:subspacesnew}
  Let $V$ be a vector space over a field~$F$, let $I$ be a finite set, and let
  $S_i \subseteq V$ for $i \in I$ be subsets such that
  $\dim \Span \bigl(\bigcup_{j \in J} S_j\bigr) \ge \#J$ for each subset $J \subseteq I$.
  Then there exists a tuple in $\prod_{i \in I} S_i$ whose components are linearly independent.
\end{lemma}

\begin{prop}\label{prop:SumsetContainsSubspacenew}
Let $n \ge 1$.
Let $I$ be a finite set.
Suppose that $0 \in S_i \subseteq \F_2^n$ and $\#S_i \ge 2$ for each $i \in I$.
If each tuple of vectors in $\prod_{i \in I} S_i$ is linearly dependent, then there exists \(J \subseteq I\) such that  \(\sum_{j \in J} S_j\) is a nontrivial \emph{subspace} of~\(\F_{2}^n\).
\end{prop}

\begin{proof}
  Let $(v_i) \in \prod_{i \in I} S_i$ be such that $\dim \Span(v_i : i \in I)$
  is maximal, say equal to $m$. 
  Let $I_m \subseteq I$ be such that $(v_i)_{i \in I_m}$ is a basis of this span, so $\#I_m=m$.
  Since $(v_i)_{i \in I}$ is dependent, $I_m \ne I$.
  Let $I_{m+1}$ be a set of size $m+1$ such that $I_m \subsetneq I_{m+1} \subseteq I$.
  No tuple in $\prod_{i \in I_{m+1}} S_i$ is linearly independent,
  so by the contrapositive of Lemma~\ref{lem:subspacesnew} for $I_{m+1}$, there exists $J \subseteq I_{m+1}$ such that
  \[ \dim \Span(\bigcup_{j \in J} S_j) \le \#J - 1 . \]
  In particular, $J$ is nonempty. Now, 
  \begin{align*} 
  \#J - 1 \le \# (J \intersect I_m) & = \dim \Span(v_j : j \in J \intersect I_m)\\
  & \le \dim \Span(v_j : j \in J)
             \le \dim \Span(\bigcup_{j \in J} S_j) \le \#J - 1 , 
    \end{align*}
  which gives the final equality in 
  \[ \Span(v_j : j \in J) = \sum_{j \in J} \{0,v_j\}
                            \subseteq \sum_{j \in J} S_j
                          \subseteq \Span(\bigcup_{j \in J} S_j)
                          = \Span(v_j : j \in J), \]
  so $\sum_{j \in J} S_j$ is a subspace of~$\F_2^n$.
  Since $J$ is nonempty and $S_j$ contains a nonzero vector for each~$j$, this subspace is nontrivial.
\end{proof}

    \begin{thm}
    \label{thm:sharpboundnew}
    Let $n \ge 1$.
    Suppose that $S_i \subseteq \F_2^n$ and $\#S_i \ge 2$ for $i=1,\ldots,t$.
    \begin{enumerate}[label* = \normalfont(\alph*)]
    \item \label{I:sumset omits one vector}
    If $\sum S_i = \F_2^n \setminus \{v\}$ for some $v \in \F_2^n$, then $t \le n-1$.
    \item \label{I:example of sum S_i}
    Let $e_1,\ldots,e_n$ be the standard basis of $\F_2^n$.
    Let $v=\sum e_i$.
    Let $S_i=S \colonequals \{0,e_1,\ldots,e_n\}$ for $i=1,\ldots,n-1$.
    Then $\sum S_i = \F_2^n - \{v\}$.
    \end{enumerate}
    \end{thm}

    \begin{remark}
    The example in \ref{I:example of sum S_i} shows that the inequality in \ref{I:sumset omits one vector} is sharp.
    \end{remark}

    \begin{remark}
    By replacing $S_1$ in \ref{I:example of sum S_i} by $S_1+v$, we obtain an example with $\sum S_i = \F_2^n - \{0\}$.
    \end{remark}

\begin{proof}[Proof of Theorem~\ref{thm:sharpboundnew}]
\hfill
\begin{enumerate}[label* = \normalfont(\alph*)]
\item For each $i$, translate $S_i$ to assume that $0 \in S_i$.
Suppose that $t \ge n$.
Then any tuple in $\prod_{i=1}^t S_i$ is linearly dependent,
since otherwise it would be a basis, contradicting $\sum S_i = \F_2^n \setminus \{v\}$.
By Proposition~\ref{prop:SumsetContainsSubspacenew}, $\sum S_i$ is a union of cosets of a nontrivial subspace of $\F_2^n$, again contradicting $\sum S_i = \F_2^n \setminus \{v\}$.
Thus $t \le n-1$.
\item
Each $S_i$ is the set of vectors with at most one nonzero coordinate,
so $\sum_{i=1}^{n-1} S_i$ is the set of vectors with at most $n-1$ nonzero coordinates.\qedhere
\end{enumerate}
\end{proof}


\end{document}